\documentclass[11pt,a4paper,reqno,twoside]{amsart}
\usepackage{amssymb}
\usepackage{amsthm}
\usepackage{xspace}
\usepackage{amsmath}
\usepackage{setspace}
\usepackage{amscd}
 
 \usepackage{tikz}
 
\usepackage{setspace}
\usepackage{enumerate}
\usepackage[english]{babel}
\usepackage[small,nohug,heads=vee]{diagrams}
\diagramstyle[labelstyle=\scriptstyle]

\newtheorem{theorem}{Theorem}[section]
       \newtheorem{lemma}[theorem]{Lemma}
       \newtheorem{proposition}[theorem]{Proposition}
       
    \newtheorem{definition}[theorem]{Definition}

\newcommand{\Z}{\mbox{$\mathbb{Z}$}}

\newcommand{\Q}{\mbox{$\mathbb{Q}$}}

\begin{document}
\title[Gaussian Mersenne Primes]{Gaussian Mersenne Primes  of the form $x^2+dy^2$  }
\author[Sushma Palimar and  Ambedkar Dukkipati]{Sushma Palimar and  Ambedkar Dukkipati}
\email{sushmapalimar@gmail.com \\sushma@csa.iisc.ernet.in \\ ad@csa.iisc.ernet.in}
\address{Dept. of Computer Science \& Automation\\Indian Institute of Science, Bangalore - 560012}

\maketitle
 \begin{abstract}
 In this paper we study  Gaussian ring $\Z[i]$ with a focus on representing Gaussian Mersenne primes $G_p$ in the form $x^2+7y^2$.
Interestingly  when such a form exists, one can observe that,   $x\equiv \pm 1\pmod{8}$ and $y\equiv 0\pmod{8}$.
To prove this property of Gaussian Mersenne primes, we show that Gaussian Mersenne primes splits completely in 
the  cyclic quartic unramified  extension of $\Q(\sqrt{-14})$ and have a trivial Artin symbol in this extension.
We generalize this result for $d\equiv 7\pmod{24}$.
We also attempt to give an alternate proof  using Artin's reciprocity law,
which was earlier given by H. W. Lenstra and P. Stevenhagen to prove a similar property on ordinary Mersenne Primes.
\end{abstract}

\footnote{Research by the first  author
is supported by National Board for Higher Mathematics-Post Doctoral Fellowship, Government of India.}


\section{Introduction}	
The study of prime numbers in the form  $x^2+dy^2$  for  
a fixed integer $d>0$ is ancient and formal study of this problem  began with  Fermat, 
later continued by Euler, which led him to discover quadratic reciprocity and many conjecture on $x^2+dy^2$  for $d>3$.
The book ~\cite{6} is a good reference for this.
Study of  Mersenne primes in the form $x^2+dy^2$, for  $d\in
[2,48]$ is carried out in detail in
\cite{Jansen:2002:mersenneprimes}.  
In \cite{henstev:2000:artinreci} Lenstra and Stevenhagen prove an interesting property of Mersenne primes 
which was first observed by Lemmermeyer.
This observation is given in the form of Theorem below. Here 
authors make use of Artin's reciprocity to demonstrate this property of Mersenne primes.
\begin{theorem}\label{mpthm}
Let $M_p$  be a Mersenne prime with   $p\equiv 1\pmod{3}$.
Write $M_p=x^2+7y^2$  with $x,y\in \Z$. Then $x$ is divisible by  $8$.
 \end{theorem}
In \cite{pal:2012:jis}, authors call $\alpha=(2+\sqrt{2})$ and the norm $N  (\frac{\alpha^{p}-1}{\alpha-1})$ is found.
In this case whenever the norm $N(\frac{\alpha^{p}-1}{\alpha-1})$ is a rational prime, then those primes are 
quadratic residues of $7$ hence can be written in the form 
$x^{2}+7y^{2}$ and it is proved using Artin's reciprocity that, $x$ is divisible by $8$ and $y\equiv \pm 3\pmod{8}$.

 In Section \ref{Gp}  we  study  Gaussian Mersenne primes 
in the form $x^2+dy^2$ with a special focus on  the form $x^2+7y^2$. 
Here we prove that whenever $G_p=x^2+7y^2$ then $x\equiv \pm 1\pmod{8}$ and $y\equiv 0\pmod{8}$
and  generalize this result  
for  all $d\equiv 7\pmod{24}$ in the next section. 
\subsection{ Gaussian Mersenne Primes}Gaussian Mersenne primes and their primality test  was first studied by 
Berrizbeitia and  Iskra  in \cite{1}.   
A Gaussian Mersenne number is an element of $\Z[i]$
given  by $\mu_p=(1\pm i)^{p}-1$, for some rational prime $p$.
The Gaussian Mersenne norm  $G_p=2^{p}-\left(\frac{2}{p}\right)2^{\frac{p+1}{2}}+1$ 
is  the norm of $\mu_p$,  $N(\mu_p)$. If $G_p$ is a 
 rational prime then we call  $G_p$  a Gaussian Mersenne prime.
 \subsection{Some results from Class Field Theory}
 \label{cft}
 We take a short and quick revision of class field theory to fix   notations.
  Gal(L/K) is  the Galois group of the field extension $K\subset L$. $\mathcal{O}_{K}$, 
the ring of algebraic integers in a finite extension $K$ of $\Q$.
 We denote $I_{K}(\Delta(L/K))$ to  be the group of all fractional ideals of $\mathcal{O}_{K}$   prime to $\Delta(L/K)$.
 The \emph{Artin symbol}  is denoted  by $((L/K)/\mathfrak{P})$, for  $\mathfrak{P}$, a prime of  $L$ containing a prime 
 $\mathfrak{p}$ of $\mathcal{O}_{K}$. Then for any $\alpha$ in  $\mathcal{O}_{L}$ 
 \begin{displaymath}
 \left( \frac{L/K}{\mathfrak{P}}\right)(\alpha)\equiv
\alpha^{N(\mathfrak{p})} \pmod{\mathfrak{P}},
\end{displaymath}
When $K\subset L$ is an Abelean extension, the Artin symbol can be written as $((L/K)/\mathfrak{p})$.
\begin{definition}
 Let $K$ be a number field and $\mathfrak{m}$  a modulus of $K$. We define $P_{K,1}(m)$ 
 to be the subgroup of $I_K(m)$ generated by the principal ideals
 $\alpha.\Z_K$, with $\alpha\in \Z_{K}$, satisfying $\alpha\equiv 1\bmod{ \mathfrak{m}_0}$, 
 where $\mathfrak{m}_0$ is a nonzero $\mathcal{O}_{K}$ ideal and 
 $\sigma(\alpha)>0$ for all infinite primes
 $\sigma$ of $K$ dividing $\mathfrak{m}_\infty $, is a product of  
   distinct real infinite primes of $K$. Thus $Cl_{m}$ is the group $I_{K}(m)/P_{K,1}(m)$.
\end{definition}
Given an order $\mathcal{O}$ in a quadratic field $K$, the index $f=[\mathcal{O}:\mathcal{O}_{K}]$ 
 is called   conductor of the order.  The discriminant  of  $\mathcal{O}$ is \(D=f^2d_k\),  $d_k$ is the discriminant of the
 maximal order $\mathcal{O}_{K}$.
\begin{definition}
 Let $K$ be a number field and $R=R_{m}(K)$ the ray class field of $K$ with modulus $m$.
 The ray class field of conductor $\mathfrak{m}=(1)$ is the Hilbert class field $H=H_1$ of $K$. 
\end{definition}
 \begin{definition}[Ring Class Field]\label{isodef}
 If $\mathcal{O}$ is the order of $\mathcal{O}_{K}$ of  conductor  $f$, we define ring class group to be 
 $I_K(f\mathcal{O}_K)/P_{K,\mathbb{Z}}(f\mathcal{O}_K)$, which is naturally isomorphic to the group of 
 ideals prime to $f$ modulo the principal ideals prime to $f$. The ring class
 field of the maximal order $\mathcal{O}_{K}$ is the Hilbert class
 field of $K$. 
 \end{definition}
  Here, the  bottom group $P_{K,\mathbb{Z}}(f\mathcal{O}_K)$ consists  of
 the principal ideals which admit a generator $\alpha$ congruent to
 some integer $a$, relatively prime to $f$,  
 i.e., \[ \alpha\equiv a\pmod{f\mathcal{O}_{K}}\] 
 We have $m=[\mathcal{O}_{K}:\mathcal{O}]=1,2,3,4,6$ if and only if 
$P_{K,1}(m)=P_{K,\mathbb{Z}}(m)$.  
This is because, $m=1,2,3,4,6$ is nothing but $\phi(m)\leq 2$.
And these are the cases where the global units 1 and --1 `fill up' $(\Z/m\Z)^*$,
and the ray class field equals the ring class field.
 \section{Gaussian Mersenne primes as $x^2+dy^2$} \label{Gp}
 In this section we derive two  properties of Gaussian Mersenne primes and prove our main result that, 
whenever $G_{p}=x^{2}+7y^{2}$ then  $x\equiv \pm 1\pmod{8}$ and $y\equiv 0\pmod{8}$.  
 \begin{definition}
  We define $F_d$ to be the quadratic form $x^2+dy^2\in \Z[x,y]$. Let
  $p$ be a prime then $p$ is represented by  
  $F_d$ if there exists $x>0$ and $y>0$ in $\Z$ such that $p=x^2+dy^2$.
 \end{definition}
\begin{proposition}\label{drem3}
 Let $d\equiv 3\bmod{4}$ be a square-free integer. Suppose
 $L=R_{2}(\Q(\sqrt{-d}))(\sqrt{2})=H(\Q(\sqrt{-2d}))$. Let 
 $G_{p}$ be a Gaussian Mersenne prime unramified in $L$ then  $F_d$
 represents $G_{p}$ if and only if  $F_{2d}$ represents $G_{p}$.
\end{proposition}
This property of Gaussian Mersenne primes stated in Proposition \ref{drem3} is obeyed by Mersenne primes too, and the proof
 is similar to the proof given in  \cite{Jansen:2002:mersenneprimes}. 
\begin{proof}
 We have
 $G_{p}=2^{p}-\left(\frac{2}{p}\right)2^{\frac{p+1}{2}}+1$. For
 $p>3$, $G_{p}\equiv1\bmod{8}$. Hence, 
$G_{p}$ splits completely in $\Q(\sqrt{2})$ and 
the prime $G_{p}$ is unramified in $L$,
 which implies $G_{p}\nmid 2d$. We have,
 $[\mathcal{O}_{K}:\Z[\sqrt{-d}]]=2$, since $d\equiv 3\pmod{4}$ 
is a positive square-free integer. Now, by using the fact $G_{p}$
splits in $\Q(\sqrt{2})$ we get $F_{d}$  
 represents $G_{p}$ if and only if $G_{p}$ splits completely in
 $R_{2}(\Q(\sqrt{-d}))(\sqrt{2})$. 
 For $K=\Q(\sqrt{-2d})$ we have $[\mathcal{O}_{K}:\Z[\sqrt{-2d}]=1$. 
Thus, $F_{2d}$ represents $G_{p}$ if
 and only if $G_{p}$ splits completely in $H(\Q(\sqrt{-2d}))$.
 By assumption we have
 $R_{2}(\Q(\sqrt{-d}))(\sqrt{2})=H(\Q(\sqrt{-2d}))$ hence the result. 
 \end{proof}
 The property of Gaussian Mersenne primes in Proposition \ref{drem1} is unique only to Gaussian Mersenne primes.
 \begin{proposition}\label{drem1}
 Let $d\equiv 1\bmod{4}$ be a square-free integer. Suppose  $L=H(\Q(\sqrt{-d}))=H(\Q(\sqrt{-2d}))$. Let
 $G_{p}$ be a Gaussian Mersenne prime unramified in $L$ then  $F_d$ represents $G_{p}$ if and only if 
 $F_{2d}$ represents $G_{p}$, which is not the case for usual Mersenne primes.
\end{proposition}
\begin{proof}
 Proof follows from the previous proposition, except for the fact,  $[\mathcal{O}_{K}:\Z[\sqrt{-d}]]=2$,
 which is $1$ in the present case, as $d\equiv 1\bmod{4}$.
\end{proof}
 \subsection{Gaussian Mersenne primes as $x^2+7y^2$}
 \label{impth}
  Here we give first few examples of Gaussian Mersenne primes and their representation as $x^2+7y^2$.
 \[G_{7}=113=1+7\cdot4^{2}\] 
\[G_{47}=140737471578113= 5732351^{2}+7\cdot3925696^{2}\] 
\[G_{73}= 9444732965601851473921=96890022433^{2}+7\cdot2854983576^{2} \]
$G_{113}=10384593717069655112945804582584321=$ \[79288509938147361^{2}+7\cdot24195412519312600^{2} \]
One can check that, Gaussian Mersenne prime $G_p$ is $1$  $\text{ mod }{8}$  for all $p>3$.
Also ,  \[\text { if }p\equiv 1\pmod{6} \text{ then } Gp\equiv1\pmod{7} \] and 
\[\text{ if }p\equiv 5\pmod{6} \text{  then }Gp\equiv4\pmod{7} \text{ for all }p.\]
For $p>7$, a  simple computation  shows that, in each case $x\equiv \pm 1 \bmod{8}$ and $y\equiv 0\bmod{8}$.  
Also, from the above table it is not difficult to see  that $y$ is exactly  divisible by $4$.
 These two observations are proved as a lemma below.
 \begin{lemma}\label{mainlemma}
  If  $G_{p}$ is represented in the form  $x^2+7y^2$ then  $x\equiv \pm1\pmod{8}$ and $4|y$.
 \end{lemma}
 \begin{proof}
 From the above discussion it is clear that, $G_p$ can be written in the form $x^2+7y^2$ if and only if 
   $p\equiv\pm 1\pmod{8}$. Also, 
 $G_{p}=2^{p}-2^{\frac{p+1}{2}}+1\equiv 1\pmod{8}$ for all $p>3$. 
  We first show that, $y$ is divisible by $4$. For this we show that, $x$ is an odd integer and $y$ is an even integer. For if $x$ is even and $y$ is odd, 
  then $x^{2}\equiv 0 \text{ or } 4\pmod{8}$
  and $y^{2}\equiv 1\pmod{8}$, which implies $x^{2}+7y^{2}\equiv\ 7\text{ or }3\pmod{8}$, a contradiction.
  Hence $x$ should be an odd integer and $y$ should be an even integer.
  Now, clearly $x^{2}\equiv 1\pmod{8}$ and $y^{2}\equiv 0 \text{ or } 4\pmod{8}$.
  If $y^{2}\equiv 4\pmod{8}$ then $x^{2}+7y^{2}\equiv 5\pmod{8}$, again a contradiction.
  Hence $y\equiv 0\pmod{4}$ and we have $4|y$. Now we show that $x\equiv\pm 1\pmod{8}$.
  Clearly $x^2+7y^2\equiv x^2\pmod{16}$. Since $x$ is an odd integer either $x^2\equiv 1\pmod{16}$ or $x^2\equiv 9\pmod{16}$. 
  If $x^2\equiv 9\pmod{16}$ then $G_{p}\equiv 9\pmod{16}$, a contradiction as $G_{p}\equiv1\pmod{16}$, for all $p\equiv\pm1\pmod{8}$.
  Thus, $x^{2}\equiv1\pmod{16}$. Hence finally $x\equiv \pm 1\pmod{8}$ and $y\equiv0\pmod{4}$.
 \end{proof}
 Now it remains to prove that $y$ is divisible by $8$ in the expression $G_p=x^2+7y^2$.
 As a first step we transform the base 
 field $\Q(\sqrt{-7})$ into $\Q(\sqrt{2})$.
 Then show that $Gp$ splits completely in the cyclic quartic unramified extension of $\Q(\sqrt{-14})$.
 This is because,  $G_{p}\equiv 1\pmod{8}$ for all $p>3$ which shows that, $G_p$ is a prime in $\Q$ which splits in $\Q(\sqrt{2})$.
\begin{proposition} \label{mainprop}
  Let $p>7$ and $G_p=2^{p}-(\frac{2}{p})2^{\frac{p+1}{2}}+1$ be a Gaussian Mersenne prime. Then for $p\equiv\pm 1\pmod{8}$
 the form $G_p= x^2+7y^2$ exists. Suppose that there exists a cyclic extension $H_4$ of
$S=\Q(\sqrt{-14})$ with $[H_4:S]=4$, $H_4\subset H(S)$ and
$\sqrt{2}\in H_4$. Then $G_p$ splits completely in $H_4$.
 \end{proposition}
\begin{proof}
 Let $K=\Q(\sqrt{-7})$ and let $J=\Q(\sqrt{2})$. As $\sqrt{2}\in H_4$, consider an extension $J=\Q(\sqrt{2})\subset H_4$
 and $Gal(H_4/\Q)$ is isomorphic to the dihedral group with $8$
 elements. Thus, there are two conjugate field extensions  
 of $J$, say $J_1$ and $J_2$ contained in $H_4$, namely $J_1=J(\sqrt{-1+2\sqrt{2}})$ and 
 $J_2=J(\sqrt{-1-2\sqrt{2}})$. Consider the following field diagram.
 
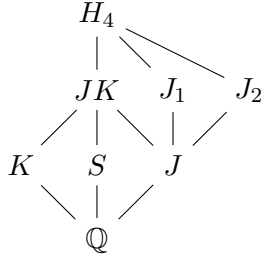
\begin{center}
 \begin{tikzpicture}
  \node (h4) at (0,2) {$H_4$};
   \node (jk) at (0,1) {$JK$};
  \node (j1) at (1,1) {$J_1$};
  \node (j2) at (2,1) {$J_2$};
  \node (k) at (-1,0) {$K$};
  \node (s) at (0,0) {$S$};
   \node (j) at (1,0) {$J$};
  \node (q) at (0,-1) {$\Q$};
  
  \draw (q) -- (k) -- (jk) -- (s) -- (q) -- (j) -- (jk)-- (h4) -- (j1) -- (j)-- (j2) -- (h4);
   
  \draw[preaction={draw=white, -,line width=6pt}] ;
\end{tikzpicture}
\end{center}  The discriminant of $JK$ over $J$ is $-7$. The discriminants $\Delta(J_1/J)$
  and $\Delta(J_2/J)$ must be relatively prime. If not there is a prime of $J$ which is ramified in $J_1$ and
  $J_2$. This prime would be ramified in $JK$, because $\Delta(H_4/JK)=1$. But then the inertia field of this prime ideals
  equals $J$, which cannot be the case  since $\Delta(H_4/JK)=1$, using the result on discriminant of tower of fields,
  and  $-7=\Delta(J_1/J)\Delta(J_2/J)$. 
Since $G_p$ splits in $\Q(\sqrt{2})$,  write $G_p=v_{p}\cdot \bar {v_p}$, identifying $v_{p}=x_1+\sqrt{2}y_1$ and $\bar {v_p}=x_1-\sqrt{2}y_1\in \Z_{J}$
for some $x_1,y_1\in \Z$.
Clearly  $\sigma(v_p)>0$ and $\sigma(\bar {v_p})>0$ for all $p$ and for all $\sigma$ in $Gal(J/\Q)$.
\[\text{ Claim: } v_p\equiv\bar v_p\equiv 1\pmod{\Delta(J_i/J)}\text { for } i=1,2  \text{ and for all } p\]
Here we consider two cases $p\equiv 1\pmod{6}$ and $p\equiv 5\pmod{6}$.
For $p\equiv 1\pmod{6}$, we have $G_p\equiv 1\pmod{7}$ hence, $v_p\equiv\bar v_p\equiv 1\pmod{\Delta(J_1/J)}$  { and }
$v_p\equiv\bar v_p\equiv 1\pmod{\Delta(J_2/J)} $. 
Since  $\Delta(J_1/J)$ and $\Delta(J_2/J)$ are coprime, we get $v_p\equiv\bar {v_p}\equiv 1\pmod{\Delta(J_i/J) }$ for $i\in {1,2}$.
Applying Artin map,
\[\left(\frac{J_i/J}{v_{p}}\right)=\left(\frac{J_i/J}{\bar {v_{p}}}\right)=1 \text{ for } i\in {1,2},\]
 since $v_p$ and $\bar v_p$ are in the kernel.
The Galois group of $H_4/J$ is isomorphic to $Gal(J_1/J)\times Gal(J_2/J)$. So
\[\left(\frac{H_4/J}{v_{p}}\right)=\left(\frac{H_4/J}{\bar {v_{p}}}\right)=1.\]
So, $G_p$ splits completely in $H_4$ for $p\equiv 1\pmod{6}$. 
For $p\equiv 5\pmod{6}$ to show that Artin symbol corresponding to $J_1/J$ and $J_2/J$ is trivial, 
the strategy has to be some what different.
We know that, for $p\equiv 5\pmod{6}$, $v_p\cdot\bar v_p\equiv 1\pmod{8}$, which is a prime in $\Q$ splits in $J=\Q(\sqrt{2})$,
and therefore $J$ completed at $v_p$ is nothing but $\Q$ completed at $G_p=v_p\cdot \bar v_p$.
Since $G_p\equiv 1\pmod{8}$,  we have, \[\left(\frac{-1}{G_p}\right)=1\] Now it remains to compute the Legendre symbol $\left(\frac{-7}{G_p}\right).$
If $p\equiv 5\pmod{6}$ then $G_p\equiv 4\pmod{7}$, hence $\left(\frac{G_p}{7}\right)=1$. Therefore by quadratic reciprocity we have,

\[\left(\frac{-7}{G_p}\right)=1,\]
and the claim follows.
 \end{proof}
 From Proposition~\ref{mainprop} and Lemma~\ref{mainlemma} it is clear
that  $G_{p}$ splits completely in the cyclic extension of
$\Q(\sqrt{-14})$. 
Now we are ready to prove our main result that, for  $p>7$ and  $p\equiv\pm 1\pmod{8}$ in the representation of $G_p$ as
$G_p=x^2+7y^2$ for some $x,y\in \Z$,  $y$ is divisible by $8$.
\begin{theorem} \label{mainth}
 Let $p>7$ and $G_p=2^{p}-(\frac{2}{p})2^{\frac{p+1}{2}}+1$ be a Gaussian Mersenne prime. Then for $p\equiv\pm 1\pmod{8}$
 the form $G_p=x^2+7y^2$ exists for some $x,y\in \Z$ and $y$ is divisible by $8$.
\end{theorem}
\begin{proof}
Here we make use of the fact that, $G_{p}$ splits completely in the cyclic extension of $\Q(\sqrt{-14})$, so
 we refer to Proposition~\ref{mainprop} for  rest of the proof. 
Consider the following lattice of fields.
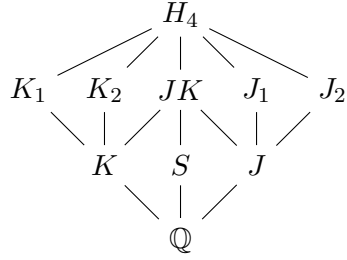
\begin{center}
 \begin{tikzpicture}
  \node (h4) at (0,2) {$H_4$};
  \node (k1) at (-2,1) {$K_1$};
  \node (k2) at (-1,1) {$K_2$};
  \node (jk) at (0,1) {$JK$};
  \node (j1) at (1,1) {$J_1$};
  \node (j2) at (2,1) {$J_2$};
  \node (k) at (-1,0) {$K$};
  \node (s) at (0,0) {$S$};
   \node (j) at (1,0) {$J$};
  \node (q) at (0,-1) {$\Q$};
  
  \draw (q) -- (k) -- (k1) -- (h4) -- (k2) -- (k) -- (jk) -- (s) -- (q) -- (j) -- (jk)-- (h4) -- (j1) -- (j)-- (j2) -- (h4);
   
  \draw[preaction={draw=white, -,line width=6pt}] ;
\end{tikzpicture}
\end{center}
 It is clear that the Galois group of $H_4/\Q$ is the dihedral group
 with $8$ elements. Here $K=\Q(\sqrt{-7})$, 
$J=\Q(\sqrt{2})$, $S=\Q(\sqrt{-14})$ and
 $\omega=\frac{1+\sqrt{-7}}{2}$. Let $\bar \omega$  be the conjugate of
 $\omega$ then $\omega\bar\omega=2$ and
$2$ splits completely in $K$ as $2=(2,\omega)(2,\bar\omega)$.  Now, $K_1$
 and $K_2$ are the two  conjugate field extensions of $K$, and $J_1$ and $J_2$
 are the two conjugate field extensions of $J$ respectively. 
Discriminant of the extension fields are,
$\Delta(S/\Q)=-56$, $\Delta(K/\Q)=-7$, $\Delta(J/\Q)=8 $, $\Delta(JK/S)=1$, $\Delta(H_4/JK)=1$,
$\Delta(JK/J)=-7$  and  $\Delta(J_1/J)\Delta(J_2/J)=-7$. Clearly $\Delta(K_1/K)$ and $\Delta(K_2/K)$ are 
coprime, and  $\Delta(K_1/K) \cdot \Delta(K_2/K)=8$.
Since the discriminants are coprime and $K_1$ and $K_2$  are conjugates, $\Delta(K_1/K)=(2,\omega)^{3}$ and
$\Delta(K_2/K)=(2,\bar\omega)^{3}$. Now, $\Delta(H_4/\Q)=8^2$ and $\Delta(K_i/K)|8^2$.
Let $\Delta=\Delta(K_1/K)$. Then there is a
surjective group homomorphism:\[Cl_{\Delta}(K)\rightarrow
Gal(K_2/K).\] 
The group $I_{K}{(\Delta)}/P_{K,1}(\Delta)=P_{K}(\Delta)/P_{K,1}(\Delta)$ is a subgroup of $Cl_{\Delta}(K)$ and 
contains the principal prime ideal $\pi=(x+y\sqrt{-7})$ of $K$.
Any prime of $\Q$, which is inert in both $J$  and $K$ splits in $S$. For example, since $7\equiv1\pmod{3}$,
the prime $3$ of $\Q$ is inert in $J$ and $K$ and splits in $S$. Thus
the decomposition group of a prime $\mathfrak{p}_3$ in $H_4$ above $3$ is $S$. 
The prime $3\Z_{K}$ of $K$ does not split in $K_1$. Hence,
$P_{K}(\Delta)/P_{K,1}(\Delta)$ maps surjective on $Gal(K_1/K)$. 
The map from $(\Z_{K}/\Delta)^{*}$ to $P_{K}(\Delta)/P_{K,1}(\Delta)$
is surjective and  the group $(\Z_{K}/\Delta)^{*}$ is isomorphic to the group
$(\Z/8\Z)^{*}$. We have the following group homomorphism,
\[(\Z/8\Z)^{*}\simeq(\Z_{K}/\Delta)^{*}\rightarrow P_{K}(\Delta)/P_{K,1}(\Delta)\rightarrow Gal(K_1/K).\]
Now $\{\pm 1\}$ is contained in the kernel of the $(\Z_{K}/\Delta)^{*}\rightarrow P_{K}(\Delta)/P_{K,1}(\Delta)$ and this map is surjective hence,
the kernel of $(\Z_{K}/\Delta)^{*}\rightarrow P_{K}(\Delta)/P_{K,1}(\Delta)$ is $\{\pm 1\}$. Thus $P_{K}(\Delta)/P_{K,1}(\Delta)$ has two elements.
So,\[(\Z/8\Z)^{*}\simeq(\Z_{K}/\Delta)^{*}\rightarrow P_{K}(\Delta)/P_{K,1}(\Delta)\simeq Gal(K_1/K).\]
From Proposition~\ref{mainprop} it is clear that $G_{p}$ splits
completely in $H_4$ and $H_4$ contains $\sqrt{2}$. 
For any prime ideal $\pi$ of $K$, $\left(\frac{K_1/K}{\pi}\right)=1$. Hence, $\pi=x+y\sqrt{-7}$ is identity in $P_{K}(\Delta)/P_{K,1}(\Delta)$.
Thus,
\begin{equation}
  \label{maineq}x+y\sqrt{-7}\equiv\pm 1\bmod{\Delta}.
\end{equation}
From Lemma~\ref{mainlemma}, we have $4|y$. We assume that, $p>7$, so
$G_{p}\equiv 1\pmod{32}$. Since, $4|y$ and $x\equiv\pm 1\pmod{8}$, we have,
\[x^2+7y^2\equiv 1\pmod{16}\] and \[ (y\cdot\sqrt{-7})^{2}\equiv 0\pmod{16}.\]
That is $y\cdot\sqrt{-7}\equiv 0 \text{ or }4\pmod{8}$. 
 If $y\cdot\sqrt{-7}\equiv   4\bmod{8}$, then there is a contradiction to 
 $\pi\equiv x+y\sqrt{-7}\equiv\pm 1\bmod{\Delta}$ in \eqref{maineq}
 and to the splitting of $\pi$ in $H_4$.  Hence, \begin{equation}\label{eq1}
                                                  y\cdot\sqrt{-7}\equiv 0  \bmod{8}   
                                                 \end{equation}
Now, for    $K=\Q(\sqrt{-7})$, let $\mathcal{O}_K$ be its ring of integers. The element   $\omega=\frac{1+\sqrt{-7}}{2}$ is an
integer of $K$, which  is a zero of the polynomial $X^2-X+2$.
As $-2\in \Z/8\Z$ satisfies $X^2-X+2=0$, we obtain a  ring homomorphism  
\[ \mathcal{O}_{K}\rightarrow \Z/8\Z\]
\[a+b\omega\rightarrow (a-2b)\bmod{8}.\] 
Now $\sqrt{-7}=2\omega-1$ maps to $3$, confirming $y\equiv 0\pmod{8}$ from equation \ref{eq1}. This completes the proof.
\end{proof}
\subsection{Artin's reciprocity law }
  Now we give an alternate proof using the similar techinique used in \cite{henstev:2000:artinreci} to prove Theorem \ref{mainth}.   
\subsection{Primes in a Quadratic field} 
In this section we refer mainly to \cite{henstev:2000:artinreci}.
To prove the results in the previous section, we need to consider the field $\Q(\sqrt{-7})$.  
We know that, for   $K=\Q(\sqrt{-7})$,  we obtain a  ring homomorphism  
\[ \mathcal{O}_{K}\rightarrow \Z/8\Z\]
\[a+b\omega\rightarrow (a-2b)\bmod{8}.\] 
kernel $\mathfrak{a}$ of the map is generated by $8$ and $\omega+2$.
Thus $\mathfrak{a}=({\omega+2})\mathcal{O}_{K}$. Their kernels are prime ideals of index $2$ in $\mathcal{O}_{K}$ with generators
$\omega$ and $\bar \omega$. Also, the identity
$\omega\bar\omega=2$ and $\omega+2=-\omega^{3}$ show that the ideal 
$\mathfrak{a}$ factors as the cube of prime $\omega\mathcal{O}_{K}$.
Now consider an Abelian extension $L$ of $K$ with Galois group $G$, say 
$L=K[\beta]$ where $\beta$ is a zero of $x^2-\omega x-1$ and $L$ has dimension $2$ over $K$. $ \text{ with } \beta^{2}-\omega\beta-1=0.$
 Now consider \[K=\Q[\sqrt{-7}]=\Q[\omega] \text{ with } \omega^{2}-\omega+2=0 \text{  and } \]
  \[L=K[\beta] ,\qquad  \beta^2-\omega\beta  -1 \]
The discriminant of $L/K$ is $\omega^2+4=\omega+2$ is non-zero 
and $L$ has dimension $2$ over $K$, hence    $L$ is  Abelian
 over $K$ with  Galois group $G$ of order $2$, say $G=\{1,\rho\}$. 
 The non-identity element $\rho$ of $G$ satisfies $\rho(\beta)=\omega-\beta=-1/\beta$.
 The discriminant $\omega+2=-\omega^3$ of the polynomial defining $L$ is not a square in $K$, hence we have $L=K[\sqrt{-\omega}]$.
 
The inclusion map ${\mathcal{O}_{K}}\rightarrow{\mathcal{O}_{L}}$ induces a ring homomorphism 
$k(\mathfrak{p})\rightarrow {\mathcal{O}_{L}}/{\mathfrak{p}\mathcal{O}_{L}}$, where $k(\mathfrak{p})=\mathcal{O}_{K}/\mathfrak{p}$, for a 
prime $\mathfrak{p}$ of $K$.
\subsection{ An outline of the proof via Artin Symbol }If $\mathfrak{p}=\pi{\mathcal{O}}_{K}$ is a prime of  $K=\Q[\sqrt{-7}]=\Q[\omega]$ different from
$\omega{\mathcal{O}}_{K}$, then the Artin symbol
$\left(\frac{L/K}{\mathfrak{p}}\right)=1$ or $\rho$ according as $\pi$ maps to $\pm 1$ or to $\pm 3$ under the map 
${\mathcal{O}}_{K}\rightarrow \Z/8\Z$ that sends $\omega$ to $-2$.
For example, $\sqrt{-7}=2\omega-1$ maps to $3\bmod{8}$ and the number $(8k\pm1 )\pm8l\sqrt{-7}$ maps to
 $\pm 1\pmod{8}$ for $l,k\in \Z$. This is because, $\sqrt{-7}=2\omega-1$ and  $\omega$ maps to $-2$
under $ \Z/8\Z$.
We see that, for $p=47$,  $G_{47}=140737471578113= 5732351^{2}+7\cdot3925696^{2} $ is nothing but, 
\[5732351\pm 3925696\sqrt{-7}\equiv-1\pm 3(0)\equiv   -1\bmod{8}.\] 
Consider an element  of the form  $x+\sqrt{-7}y$ of norm $x^2+7y^2$ which is a rational prime.
The Artin symbol equals $1$ if $x+3y$ is $\pm 1\bmod{8}$, and $\rho$ otherwise. 
 From Lemma~\ref{mainlemma}, it is clear that $x$ leaves the remainder
 $\pm 1\bmod{8}$ and $y$ is divisible by $4$. 
 Hence the Artin symbol is $1$ if and only if $y$ is divisible by $8$. This observation is equivalent to the assertion that,
 \textit{any prime of $K=\Q(\sqrt{-7})$ of norm $G_{p}$ has trivial Artin symbol in the quadratic extension $L=K[\omega]$}.
 The proof works in the extension $N=K(\sqrt{-\omega},\sqrt{-\bar {\omega}})$, which is the union of two conjugate quadratic extensions
 $L=K(\sqrt{-\omega})$ and $L=K(\sqrt{-\bar {\omega}})$.  
Further details can be filled in by following  Theorem \ref{mainth} and the last theorem in  \cite{henstev:2000:artinreci} and \cite{pal:2012:jis}.
 
  \section{Generalization of   Theorem~\ref{mainth}}
 In this section we generalize the result of Theorem \ref{mainth}  for $K=\Q(\sqrt{-d})$ that has the same properties as  $K=\Q(\sqrt{-7})$.
 In this case there exists a  cyclic extension  $H_4$ of $S=\Q(\sqrt{-2\cdot d})$ and $\sqrt{2}\in H_4$, 
 so that $G_p$ splits completely in $H_4$. For this one requires $\left(\frac{2}{d}\right)=1$ and
 $\left(\frac{-d}{G_p}\right)=1$. Thus, with 
 these conditions 
we state the general form of  Proposition~\ref{mainprop}.
\begin{proposition}
Let  $d\equiv 3\pmod{4}$ be a square free integer and $d>0$.
Let $G_p$ be a Gaussian Mersenne prime  and   $G_p=x^{2}+dy^{2}$ for $d\equiv 3\pmod{4}$ 
exists and $\left(\frac{2}{d}\right)=1=\left(\frac{-d}{G_p}\right) $. Suppose that there exists a  cyclic extension $H_4$ of
$S=\Q(\sqrt{-2\cdot d})$ with $[H_4:S]=4$, $H_4\subset H(S)$ and
$\sqrt{2}\in H_4$. Then $G_p$ splits completely in $H_4$.
\end{proposition}
\begin{proof}
 Proof of this proposition is similar to  Proposition~\ref{mainprop}.
\end{proof}
Thus, the generalized form of    Theorem \ref{mainth} is stated below, proof of this theorem is similar to the proof of Theorem ~\ref{mainth}. 
\begin{theorem}
Let $d\equiv 7\pmod{24}$  be a square free integer and $d>0$.
 Suppose   $G_{p}$ is a Gaussian Mersenne prime which splits completely in a  cyclic extension $H_4$  of
$S=\Q(\sqrt{-2\cdot d})$. If $G_p=x^{2}+dy^{2}$ for $d\equiv 7\pmod{24}$  exists then 
 $8|y$.
\end{theorem}

\section{Conclusion}
   In the case of  Mersenne primes $2^{p}-1$,
     whenever $2^{p}-1$ is a quadratic residue of $7$, then  $2^p-1$ takes the form $x^2+7y^2$.
    Franz Lemmermeyer made an observation that, in this representation 
 $x$ is divisible by $8$ and $y$ leaves the remainder $\pm 3$ when
 divided by $8$. This was later proved by H. W. Lenstra and
 P. Stevenhagen using Artin reciprocity law and presented this
 result on the occasion of birth centenary of Emil Artin on 3 March, 1998 at Universiteit van Amsterdam\cite{henstev:2000:artinreci}.
 In \cite{pal:2012:jis}, authors made   similar observations for rational primes $N(\frac{\alpha^{p}-1}{\alpha-1})$, 
 for $\alpha=2+\sqrt{2}$ and proved using Artin's reciprocity law.
 In this paper, we have made an observation that, if a  Gaussian Mersenne prime is a quadratic residue of $7$, then in the representation of 
 $G_{p}$ as $x^2+7y^2$, we have  $x\equiv \pm 1\pmod{8}$
 and $y$ is divisible by $8$ and proved this result using Artin reciprocity law and generalized this result for 
 $d\equiv 7\pmod{24}$.
\section{Acknowledgement}
  The authors would like to thank Professor P. Stevenhagen,
  Universiteit Leiden, for his valuable insights of class field theory concepts. 
  The authors would also like to thank Professor U.K.Anandavardhanan, IIT Bombay, 
  for helping to improve earlier versions of this manuscript.
The first author would  like to acknowledge  National Board for Higher Mathematics,  Government  of India for funding this work through Post Doctoral Fellowship.

\end{document}